\documentclass[11pt]{amsart}
\usepackage{latexsym,amsmath,amssymb,amsfonts,amscd,graphics,appendix,amsxtra}
\usepackage[mathscr]{eucal}
\usepackage[normalem]{ulem}
\usepackage{fullpage}
\usepackage{color}
 
\newtheorem{theorem}{Theorem}[section]
\newtheorem{proposition}[theorem]{Proposition}
\newtheorem{lemma}[theorem]{Lemma}

\newtheorem{corollary}[theorem]{Corollary}

\newtheorem{D}[theorem]{Definition}
\newenvironment{definition}{\begin{D} \rm }{\end{D}}
\newtheorem{R}[theorem]{Remark}
\newenvironment{remark}{\begin{R}\rm }{\end{R}}
\newtheorem{E}[theorem]{Example}

\numberwithin{equation}{section}

\newcommand{\Dis}{\displaystyle}

\def\Q{\mathbb{Q}}

\def\Cee{\mathbb{C}}

\def\Gr{\operatorname{Gr}}

\def\Sp{\operatorname{Sp}}

\def\scrO{\mathcal{O}}

\def\spcheck{^{\vee}}
\def\hX{\widehat{X}}

\def\cX{\mathcal{X}}

\def\cE{\mathcal{E}}
\def\cL{\mathcal{L}}
\def\cY{\mathcal{Y}}

\def\hcX{\widehat{\cX}}
\def\uOb{\underline\Omega^\bullet}
\def\uOp{\underline\Omega^p}
\def\uOnp{\underline\Omega^{n-p}}

  \def\bD{\mathbb{D}}

  \title{The higher Du Bois and higher rational properties for isolated singularities}

\author[R. Friedman]{Robert Friedman}
\address{Columbia University, Department of Mathematics, New York, NY 10027}
\email{rf@math.columbia.edu}
\author[R. Laza]{Radu Laza}
\address{Stony Brook University, Department of Mathematics, Stony Brook, NY 11794}
\email{radu.laza@stonybrook.edu}

\begin{document}
\begin{abstract}
 Higher rational and higher Du Bois singularities have recently been introduced as natural generalizations of the standard definitions of rational and Du Bois singularities. In this note, we discuss these properties for isolated singularities, especially in the locally complete intersection (lci) case. First, we reprove the fact that a $k$-rational isolated singularity is $k$-Du Bois  without any lci assumption. For isolated lci singularities, we  give a complete characterization of the  $k$-Du Bois and $k$-rational singularities in terms of standard invariants of singularities. In particular, we show that $k$-Du Bois singularities are $(k-1)$-rational for isolated lci singularities.  In the course of the proof, we establish some new relations between invariants of isolated lci singularities and show that many of these vanish. The methods also lead to a quick proof of an inversion of adjunction theorem in the isolated lci case.  Finally, we discuss  some results specific to the hypersurface case. \end{abstract}
\thanks{Research of the second author is supported in part by NSF grant DMS-2101640}.
\bibliographystyle{amsalpha}
\maketitle

\section{Introduction}\label{section1}
In algebraic geometry, it is often important to identify singularities which are ``mild" from an appropriate viewpoint, e.g.\ that of Hodge theory or birational geometry. From a cohomological perspective,  rational and Du Bois singularities are classical examples of such types of singularities, and canonical and log canonical singularities form an important class from an adjunction theoretic perspective.  There is a well understood relationship between these classes of singularities (see \cite{Steenbrink}, \cite{Kovacs99}, \cite{KK}).  In particular, if a singularity is normal and Gorenstein, then it is Du Bois if and only if it is  log canonical. Recently, M. Musta\c{t}\u{a}  and M. Popa \cite{MP-Mem} initiated  the study of higher adjunction properties and introduced the notion of \textsl{$k$-log canonical}  hypersurface singularities. This led to the study of \textsl{$k$-Du Bois singularities} (\cite{MOPW} and \cite{JKSY-duBois}) and \textsl{$k$-rational singularities} (\cite{KL2}, \cite{FL,FL-DuBois}). The case $k=0$ recovers the standard notions (of Du Bois and rational singularities), while as $k$ increases the singularities become milder, leading to   Hodge theoretic behavior closer to that of the smooth case (cf.\ \cite[Cor.\ 1.4, Cor.\ 1.11]{FL-DuBois}).

The purpose of this note is to discuss these notions of singularities and associated results in the case of isolated singularities, with particular attention to isolated lci singularities, using techniques of the  ``classical" Hodge theory of singularities as developed by Steenbrink in \cite{Steenbrink} and  \cite{SteenbrinkDB}. Along the way, we find new relations between well-known invariants of singularities and are able to extend some results previously known in the hypersurface case to the case of \emph{isolated} lci singularities. The starting point of this paper was \S3 of the first version of \cite{FL}, which we expand and streamline here. The subsequent papers \cite{FL-DuBois}, \cite{MP-rat} and \cite{ChenDirksM}  deal with the general case of certain results in this paper, but the discussion of the isolated case helps to clarify and extend many of these statements.

\medskip

We begin by defining the  higher Du Bois and higher rationality properties.
\begin{definition}\label{def1} Let $X$ be a reduced scheme (over $\Cee$) or complex analytic space. For $p\ge 0$, we denote by $\Omega_X^p$ the sheaf $\bigwedge^p\Omega^1_X$ of K\"ahler differentials on $X$, and by $\uOp_X$ the $p^{\text{\rm{th}}}$ graded piece of the filtered de Rham complex $\uOb_X$ (or Deligne-Du Bois complex; cf. \cite{duBois}, \cite[\S7.3]{PS}). Note that   $\uOp_X$ exists as an object in the bounded derived category $D_b^{\text{\rm{coh}}}(X)$ of $X$, and, for every $p$,  there is a natural map $\phi^p:\Omega_X^p\to \uOp_X$. Then $X$ \textsl{has $k$-Du Bois singularities} if  
$\phi^p$ is an  isomorphism  in $D_b^{\text{\rm{coh}}}(X)$ for $0\le p\le k$, where the case $k=0$ is the usual definition of Du Bois singularities. 
\end{definition}

The related notion of \textsl{$k$-rational} singularities has been proposed (in varying degrees of generality) in \cite[\S4]{KL2}, \S3 of the first version of \cite{FL}, and in \cite{FL-DuBois}. Namely, by   \cite[Lemma 3.11]{FL-DuBois}, the universal properties of $\uOb_X$ lead to natural duality maps $\psi^p:\uOp_X\to \bD_X(\uOnp_X)$, where $\bD_X$ denotes the Grothendieck duality functor for $D_b^{\text{\rm{coh}}}(X)$. Then  \cite[Def. 3.12]{FL-DuBois} defines $k$-rational singularities generalizing the usual definition of rational singularities (the case $k=0$), and factoring through the $k$-Du Bois definition above, as follows:

\begin{definition}\label{def3} The variety $X$ has \textsl{$k$-rational singularities} if the maps 
$\Omega_X^p\xrightarrow{\psi^p \circ \phi^p}  \bD_X(\uOnp_{X})$ 
are isomorphisms in $D_b^{\text{\rm{coh}}}(X)$ for all $0\le p\le k$.
\end{definition}

 While this definition might seem hard to apply in practice, 
 the situation becomes manageable in the case of isolated singularities. Specifically, assuming that $(X,x)$ is an isolated singularity, we let $X$ be a good Stein representative and $\pi\colon \hX \to X$ be a log resolution with reduced exceptional divisor $E$. Then the $k$-Du Bois condition reads
 (\cite[Ex. 7.25]{PS}, but see also \cite[Rem. 3.19]{FL-DuBois}):

\begin{theorem}\label{thm2} Suppose that $X$ is normal and has an isolated singularity at $x$.   Then $X$  has $k$-Du Bois singularities $\iff$ for all $p\le k$, the map $\Omega^p_X \to \mathcal{H}^0\uOp_X$ is an isomorphism and,  for all $p\le k$   and all $q> 0$, $H^q(\hX; \Omega^p_{\hX}(\log E)(-E))=0$.

Moreover, if $X$ has an isolated lci singularity, then   $X$  has $k$-Du Bois singularities $\iff$  for all $p\le k$   and all $q> 0$, $H^q(\hX; \Omega^p_{\hX}(\log E)(-E))=0$. \qed
\end{theorem}

As in the case of $k$-Du Bois singularities, the situation for rational singularities is much simpler if  $X$ has isolated singularities \cite[Cor. 3.17 and Lem. 3.18]{FL-DuBois}:

\begin{theorem}\label{thm4} Suppose that $X$ is normal and has an isolated singularity at $x$. Then $X$ has $k$-rational singularities $\iff$ for all $p\le k$,   the natural map $\Omega_X^p \to R^0\pi_*\Omega^p_{\hX}(\log E)$ is an isomorphism, and $H^q(\hX; \Omega^p_{\hX}(\log E))=0$ for all $q> 0$.

Moreover, if $X$ has an isolated lci singularity, then   $X$  has $k$-rational singularities $\iff$  for all $p\le k$   and all $q> 0$, $H^q(\hX; \Omega^p_{\hX}(\log E))=0$. \qed
\end{theorem}

 In \cite{FL-DuBois}, partially generalizing a theorem of Kov\'acs \cite{Kovacs99}, we showed:

\begin{theorem}\label{thm5} Suppose that $X$ is $k$-rational and that $X$ has either isolated or lci singularities. Then $X$ is $k$-Du Bois. \qed
\end{theorem}

Musta\c{t}\u{a} and Popa \cite{MP-rat} have  given an independent proof of Theorem~\ref{thm5} for the case of lci singularities.

\medskip

For hypersurface singularities, not necessarily isolated, Saito has defined an invariant $\widetilde\alpha_{X,x}$ as follows:

\begin{definition}\label{def6} If $X$ is a hypersurface in $\Cee^{n+1}$ and $x\in X$, define $\widetilde\alpha_{X,x}$, the \textsl{minimal exponent} of $X$ at $x$, to be the smallest root of $b_f(-s)/1+s$, where $b_f$ is the Bernstein polynomial of a local defining equation for $X$ at $x$. If $X$ has a unique singular point $x$, we abbreviate  $\widetilde\alpha_{X,x}$ by $\widetilde\alpha_X$. 
\end{definition}

 The $k$-Du Bois and $k$-rational properties are related to $\widetilde\alpha_{X,x}$ as follows \cite[Thm.\ 1]{JKSY-duBois}, \cite[Thm.\ 1.1.]{MOPW}:
 
 \begin{theorem}\label{thm7} The hypersurface $X$ has   $k$-Du Bois singularities in a neighborhood of $x$ $\iff$ $\widetilde\alpha_{X,x}\geq k+1$. \qed
 \end{theorem}
 
 In the first version of \cite{FL}, we showed:
 
 \begin{theorem}\label{thm8} If $X$ has an isolated hypersurface singularity at $x$, then $X$ is   $k$-rational $\iff$ $\widetilde\alpha_X> k+1$. \qed
 \end{theorem}
 
 \begin{corollary}\label{cor9} Suppose that $X$ has an isolated hypersurface singularity at $x$. If $X$ is $k$-Du Bois, then $X$ is $(k-1)$-rational. \qed
 \end{corollary}
 
 \begin{remark} Saito (appendix to \cite{FL-DuBois}) and  Musta\c{t}\u{a} and Popa \cite{MP-rat}  have proved Theorem~\ref{thm8} and hence Corollary~\ref{cor9}
 without assuming isolated singularities.
\end{remark}

One goal of this paper is to give a short proof of Theorem~\ref{thm5} for isolated singularities and to generalize Corollary~\ref{cor9} to the isolated lci case.  A key point is to find the appropriate analogues of Theorems~\ref{thm7} and \ref{thm8} in the lci case.  Saito's invariant $\widetilde\alpha_X$, which is defined for hypersurfaces and takes on positive rational values, is replaced by two integer invariants,  $\sum_{p=0}^ks_p$ and $\sum_{p=0}^ks_{n-p}$. Here, $s_p = \dim\Gr^p_FH^n(M)$ is the dimension of the $p^{\text\rm{{th}}}$ graded piece of the Hodge filtration on the Milnor fiber $M$ of $X$, and is independent of the choice of a smoothing in the lci case.  We show the following (cf. Proposition~\ref{nminuspvsp} and Corollary~\ref{nminuspvsp2}):
 
 \begin{theorem}\label{thm10} \qquad $\Dis \sum_{p=0}^{k-1}s_{n-p}\le \sum_{p=0}^ks_p\le \sum_{p=0}^ks_{n-p}.$
 \end{theorem} 
 
 Then Theorems~\ref{thm7} and \ref{thm8} are replaced by the following (Theorem~\ref{mainDB0} and Theorem~\ref{mainrat}):
 
 \begin{theorem}\label{thm11} Let $X$ be an  isolated lci singulariity.
 \begin{enumerate}
 \item[\rm(i)] $X$ is $k$-Du Bois $\iff$  $\Dis\sum_{p=0}^ks_p =0$.
  \item[\rm(ii)] $X$ is $k$-rational $\iff$  $\Dis\sum_{p=0}^ks_{n-p} =0$
 \end{enumerate}  
 \end{theorem}
 \begin{remark}
 The case $k=0$ reads: {\it $X$ is Du Bois $\iff$ $s_0=0$,  and $X$ is rational $\iff$ $s_n=0$}. This case is due to Steenbrink and Saito. 
 \end{remark}
 Combining the two results, we immediately see: 
 
  \begin{corollary}\label{cor12} Suppose that $X$ has an isolated lci singularity at $x$. If $X$ is $k$-Du Bois, then $X$ is $(k-1)$-rational. \qed
 \end{corollary}  
 
 Chen-Dirks-Musta\c{t}\u{a} have recently proved  Corollary~\ref{cor12} without assuming that the singularities of $X$ are isolated \cite{ChenDirksM}. The basic idea of the proof is as follows.   Musta\c{t}\u{a} and Popa \cite[Theorem F]{MP-loc} have defined an integer invariant $p(X)$ in terms of filtrations and have showed that, for $X$ a local complete intersection,  $p(X) \ge k$ $\iff$ $X$ is $k$-Du Bois. There is also the analogue $\widetilde\alpha_X$ of  the minimal exponent as defined by Chen-Dirks-Musta\c{t}\u{a}-Olano  \cite{CDMO}.  If $X$ is a local complete intersection of codimension $r$, then it follows from \cite{MP-loc} that $X$ is $k$-Du Bois $\iff$   $\widetilde\alpha_X \ge k+r$ and from \cite{ChenDirksM} that $X$ is $k$-rational $\iff$   $\widetilde\alpha_X > k+ r$.  Combining these two results gives a proof of Corollary~\ref{cor12} in the non-isolated case. 
 
Although the methods of this paper only apply to the case of isolated singularities, they have the added benefit  of exhibiting  the following   strong  numerical connection between   higher Du Bois and rational singularities in the isolated lci case. For example, we show the following:
 
 \begin{theorem}\label{Thm-improvesing}  Let $(X,x)$ be an isolated lci of dimension $n$.
 \begin{enumerate}
 \item[\rm(i)] If $X$ is $k$-rational, then $X$ is $k$-Du Bois. More precisely, $X$ is $k$-rational $\iff$  $X$ is $k$-Du Bois and  $\ell^{k,n-k-1}=0$;   
  \item[\rm(ii)] $X$ is $k$-Du Bois   $\iff$  $X$ is $(k-1)$-rational and $b^{k,n-k-1}=0$,
 \end{enumerate}
 where $\ell^{p,q}$ and $b^{p,q}$ are the link and Du Bois invariants respectively (see Definition \ref{def-inv}).  \qed
 \end{theorem}


The contents of this paper are as follows. In Section~\ref{section2}, we collect some basic facts about the Hodge theory of resolutions of isolated singularities. We further study the invariants $\sum_{p=0}^ks_p$ and $\sum_{p=0}^ks_{n-p}$ in detail and prove Proposition~\ref{nminuspvsp} as in Steenbrink \cite{SteenbrinkDB}. Section~\ref{section3} gives a short proof of Theorem~\ref{thm5} in the isolated case, following the strategy of Steenbrink's proof for the case $k=0$ \cite[Proposition 3.7]{Steenbrink}. The remainder of the paper is concerned with the isolated lci case.  Section~\ref{section4} is devoted to an analysis of Steenbrink's construction of the mixed Hodge structure on the Milnor fiber. As a corollary, we obtain further relations among the basic numerical invariants (Theorem~\ref{relateDBs}). With these preliminaries, we prove the main theorems characterizing $k$-Du Bois and $k$-rational isolated lci singularities in Section~\ref{section5}. We also prove a new inverse of adjunction type result (Corollary~\ref{invadj}). In Section~\ref{section6}, we specialize to the case of isolated hypersurface singularities. Our main interest is in those which are $k$-Du Bois but not $k$-rational, or equivalently those for which $\widetilde\alpha_X= k+1$. A result of  Dimca-Saito \cite[\S4.11]{DS} shows that these singularities, which we term \textsl{$k$-liminal}, have especially appealing properties. These results, for $k=1$, are used in an essential way in \cite{FL}.

\subsection*{Notations and Conventions} We work in the analytic category. $X$ will always denote a good Stein representative for the germ of the isolated singularity $(X,x)$ of dimension $n\ge 2$, and $\pi \colon \hX \to X$ will denote a log resolution, with exceptional divisor $E$.  In particular, we will always assume that $X$ is contractible and hence that $E$ is a deformation retract of $\hX$. Let $U = X-\{x\} =\hX -E$ be the link of the singularity. Then $U$ has the homotopy type of an oriented $(2n-1)$-manifold $L$.

\subsection*{Acknowledgements} It is a pleasure to thank M. Musta\c{t}\u{a}, M. Popa, and M. Saito for stimulating correspondence during the preparation of this paper. We would also like to thank the referee for a careful reading of the paper and several helpful suggestions. 

\section{Some general results}\label{section2}

\subsection{Some basic Hodge theory} Following our conventions, $X$ is a good Stein representative of an isolated singularity of dimension $n \ge 2$, $\pi\colon \hX \to X$ is a log resolution with exceptional divisor $E$, and $U = X-\{x\} = \hX - E$. If $L$ is the link of the singularity, then   $U$ and $L$ have  the same homotopy type. We begin by recalling some results which hold in slightly greater generality before specializing to the lci case. 
 
\begin{lemma}\label{normalcrossingbasics}  Let $E$ be a simple normal crossing divisor in a complex manifold $\hX$.  \begin{enumerate}
\item[\rm(i)] Let $\tau_E^\bullet \subseteq \Omega_E^\bullet$ be the subcomplex of torsion differentials, i.e.\ the differentials whose support is contained in $E_{\text{\rm{sing}}}$, where $\Omega_E^\bullet =\bigwedge^\bullet \Omega_E^1$. Then $(\Omega_E^\bullet/\tau_E^\bullet, d)$ is a resolution of the constant sheaf $\Cee$, and in fact $\Omega_E^\bullet/\tau_E^\bullet \cong \uOb_E$. If all components of $E$ are compact K\"ahler, then the
 spectral sequence with $E_1$ term
$$E_1^{p,q} = H^q(E;\Omega_E^p/\tau_E^p) \implies H^{p+q}(E;\Cee)$$
degenerates at $E_1$ and the induced filtration on $H^{p+q}(E;\Cee)$ is the Hodge filtration.
\item[\rm(ii)] Setting $E^{[k]}$ to be the locus of $k$-fold intersections if the components of $E$, with $E^{[0]} = \coprod _i E_i$, there is an exact sequence (omitting the inclusion morphisms)
$$0 \to \Omega_E^\bullet/\tau_E^\bullet \to \Omega_{E^{[0]}}^\bullet \to \Omega_{E^{[1]}}^\bullet \to \Omega_{E^{[2]}}^\bullet \to \cdots$$
\item[\rm(iii)] The morphism $\Omega_{\hX}^\bullet \to \Omega_E^\bullet/\tau_E^\bullet$ is surjective and its kernel is $\Omega_{\hX}^\bullet(\log  E)(-E)$. \qed
\end{enumerate}
\end{lemma} 

From now on, we return to the convention that $X$ is a good Stein representative for the isolated singularity $X$, with log resolution $\hX$. Thus, for a coherent sheaf $\mathcal{F}$ on $\hX$ and for $i>0$, we can identify $R^i\pi_*\mathcal{F}$ with $H^i(\hX; \mathcal{F})$. The following is the fundamental vanishing theorem of Guill\'en, Navarro Aznar, Pascual-Gainza, Puerta and Steenbrink (see e.g.\ \cite[p. 181]{PS}): 

\begin{theorem}\label{GNAPS}  For $p+q> n$, 
$H^q(\hX; \Omega^p_{\hX}(\log E)(-E))  =0$.   \qed
\end{theorem}

 The first part of the following lemma is due  to Namikawa-Steenbrink  \cite[p.\ 407]{NS}, \cite[p.\ 1369]{SteenbrinkDB}:

\begin{lemma}\label{convzero}  Under the convention that  $E$ is a deformation retract of $\hX$, let $L$ be the link of the pair $(\hX, E)$.
\begin{enumerate}
\item[\rm(i)] The hypercohomology groups $\mathbb{H}^k(\hX; \Omega^\bullet_{\hX}(\log E)(-E))=0$ for all $k$, and hence the spectral sequence with $E_1^{p,q}= H^q(\hX; \Omega^p_{\hX}(\log E)(-E))$ converges to zero.  
\item[\rm(ii)] For all $k$, the natural map $\mathbb{H}^k(\hX; \Omega^\bullet_{\hX}(\log E)) \to \mathbb{H}^k(\hX; \Omega^\bullet_{\hX}(\log E)|E)$ is an isomorphism, and hence $\mathbb{H}^k(\hX; \Omega^\bullet_{\hX}(\log E)|E)\cong H^k(L)$. 
\item[\rm(iii)]  There is a long exact sequence
\begin{gather*}
\cdots  \to  H^q (\hX; \Omega_{\hX}^p(\log E)(-E))\to H^q(\hX; \Omega_{\hX}^p(\log E))\to 
\Gr^p_FH^{p+q}(L)\to  \\
 \to  H ^{q+1} (\hX; \Omega_{\hX}^p(\log E)(-E))\to \cdots,  
\end{gather*}
where $\Gr^p_FH^{p+q}(L)$ denotes the associated graded for the Hodge filtration of the mixed Hodge structure on $H^{p+q}(L)$.  
\item[\rm(iv)] There is a long exact sequence
$$ \cdots \to H^{k-1}(L) \to H^k_E(\hX) \to H^k(E) \to H^k(L) \to \cdots,$$
which is an exact sequence  of mixed Hodge structures.
Similarly, there are exact sequences
$$\dots \to H^q(\Omega^p_{\hX}(\log E)|E) \to H^q(\Omega^p_{\hX}(\log E)/\Omega^p_{\hX}) \to H^{q+1}(\Omega^p_E/\tau_E^p) \to H^{q+1}(\Omega^p_{\hX}(\log E)|E) \to \dots $$
\end{enumerate}
\end{lemma} 
\begin{proof}
 (i) By Lemma~\ref{normalcrossingbasics}(iii), there is  an exact sequence 
$$0 \to \Omega^\bullet_{\hX}(\log E)(-E) \to \Omega^\bullet_{\hX} \to \Omega^\bullet_E/\tau_E^\bullet\to 0.$$
Taking hypercohomology, we have a long exact sequence
$$\cdots \to \mathbb{H}^k(\hX; \Omega^\bullet_{\hX}(\log E)(-E)) \to \mathbb{H}^k(\hX; \Omega^\bullet_{\hX}) \to \mathbb{H}^k(E; \Omega^\bullet_E/\tau_E^\bullet) \to \cdots $$
By assumption, the map $\mathbb{H}^k(\hX; \Omega^\bullet_{\hX}) = H^k(\hX) \to \mathbb{H}^k(E; \Omega^\bullet_E/\tau_E^\bullet)=H^k(E)$ is an isomorphism for all $k$. Hence $\mathbb{H}^k(\hX; \Omega^\bullet_{\hX}(\log E)(-E))=0$ for all $k$. 

\smallskip
\noindent (ii) We have the  long exact hypercohomology sequence associated to 
\begin{equation}\label{eqn2.1}
0 \to \Omega_{\hX}^\bullet(\log E)(-E) \to \Omega_{\hX}^\bullet(\log E) \to \Omega_{\hX}^\bullet(\log E)|E \to 0.
\end{equation}
 By (i),  $\mathbb{H}^k(\hX; \Omega^\bullet_{\hX}(\log E)) \cong \mathbb{H}^k(\hX; \Omega^\bullet_{\hX}(\log E)|E)$. Finally,  $\mathbb{H}^k(\hX; \Omega^\bullet_{\hX}(\log E)) \cong H^k(\hX -E) = H^k(U)$. Thus $\mathbb{H}^k(\hX; \Omega^\bullet_{\hX}(\log E)|E) \cong H^k(U) \cong H^k(L)$. 
 
 \smallskip
\noindent (iii) By e.g.\ \cite[Theorem 6.9]{PS}, the  filtration on $\mathbb{H}^k(\hX; \Omega^\bullet_{\hX}(\log E)|E)$ induced by the trivial filtration on $\Omega^\bullet_{\hX}(\log E)|E$ computes the Hodge filtration of the mixed Hodge structure on $H^\bullet(L)$. In particular, $\Gr^p_FH^{p+q}(L) = H^q(\hX; \Omega^p_{\hX}(\log E)|E)$. Then (iii) follows by taking the usual cohomology of the exact sequence (\ref{eqn2.1}).

\smallskip
\noindent (iv) There is a commutative diagram
$$\begin{CD}
0@>>>  \Omega^\bullet_{\hX} /\Omega^\bullet_{\hX}(\log E)(-E) @>>> \Omega^\bullet_{\hX}(\log E)/\Omega^\bullet_{\hX}(\log E)(-E) @>>> \Omega^\bullet_{\hX}(\log E)/\Omega^\bullet_{\hX} @>>> 0\\
@. @V{\cong}VV @| @| @.\\
0 @>>> \Omega^\bullet_E/\tau_E^\bullet @>>> \Omega^\bullet_{\hX}(\log E)|E @>>> \Omega^\bullet_{\hX}(\log E)/\Omega^\bullet_{\hX} @>>> 0,
\end{CD}$$
Here $\mathbb{H}^k(E;\Omega^\bullet_{\hX}(\log E)|E) = H^k(L)$ and $\mathbb{H}^k(E;\Omega^\bullet_{\hX}(\log E)/\Omega^\bullet_{\hX}) = H^{k+1}_E(\hX)$. The long exact sequences on hypercohomology or cohomology then give the exact sequences in (iv).  For the statement about mixed Hodge structures, we refer to \cite[\S1]{Steenbrink} or \cite[Chapter 6]{PS}. 
\end{proof}

\begin{remark}    Note that $\Omega^\bullet_{\hX}(\log E)|E$ is independent of the choice of the representative $X$ representing the germ, and the corresponding cohomology and hypercohomology groups are independent of the choice of a resolution. Likewise, the mixed Hodge structure on $H^k(L)$ only depends on the germ $(X,x)$. By contrast, the mixed Hodge structure on  the cohomology $H^k(M)$ of the Milnor fiber depends on the choice of a smoothing, even in the lci case.  
\end{remark} 

As a corollary of Lemma~\ref{convzero}(i), we obtain the following, first noted by Steenbrink in case $k=0$ \cite[p.\ 1369]{SteenbrinkDB}:

\begin{lemma}[Extra vanishing lemma]\label{extravan}  Suppose that the isolated singularity $X$ is $k$-Du Bois. Then $H^{n-k-1}(\hX; \Omega^{k+1}_{\hX}(\log E)(-E)) = 0$. 
\end{lemma}
\begin{proof} In the spectral sequence of Lemma~\ref{convzero}(i), $E_1^{p,q} = 0$ for $p+q > n$ by Theorem~\ref{GNAPS}, and $E_1^{p,q} = 0$ for $p\le k$ and all $q > 0$ by the hypothesis that $X$ is $k$-Du Bois (Theorem~\ref{thm2}). Then an examination of the spectral sequence shows that all differentials $d_r$ with source or target $E_1^{k+1,n-k-1}=H^{n-k-1}(\hX; \Omega^{k+1}_{\hX}(\log E)(-E))$ are $0$. Thus 
$$E_\infty^{k+1,n-k-1} = E_1^{k+1,n-k-1}=H^{n-k-1}(\hX; \Omega^{k+1}_{\hX}(\log E)(-E)).$$
Since the spectral sequence converges to $0$, we must have $H^{n-k-1}(\hX; \Omega^{k+1}_{\hX}(\log E)(-E))=0$.
\end{proof}

Finally, we note the semipurity  theorem of Goresky-MacPherson \cite[Theorem 1.11]{Steenbrink}:

\begin{theorem}\label{semipure} The morphism of mixed Hodge structures  $H^k_E(\hX) \to H^k(E)$ is injective for $k\leq  n$, surjective for $k\geq n$, and   an isomorphism for $k=n$. Thus, for $i\leq n-1$, the map $H^i(E) \to H^i(L)$ is surjective and hence $\Gr^W_rH^i(L) = 0$ for $r >i$. \qed
\end{theorem}

\subsection{Numerical invariants}  
\begin{definition}\label{def-inv}
 Following \cite{SteenbrinkDB}, for $q>0$, we define the \textsl{Du Bois invariants} 
$$b^{p,q} = \dim H^q(\hX; \Omega^p_{\hX}(\log E)(-E))  = \dim H^{n-q}_E(\hX; \Omega^{n-p}_{\hX}(\log E)).$$
By Theorem~\ref{GNAPS}, $b^{p,q} =0$ if $p+q> n$.

There are also the  \textsl{link invariants}
$$\ell^{p,q} = \ell^{p,q}(L) = \dim  H^q(E;\Omega^p_{\hX}(\log E)|E) = \dim \Gr_F^pH^{p+q}(L).$$
By Serre duality, $\ell^{p,q} =  \ell^{n-p,n-q-1}$. 

Finally, define the Deligne-Hodge numbers of the link $L$: $h^{p,q}_i = h^{p,q}_i(L) = \dim \Gr_F^p\Gr_{p+q}^WH^i(L)$. Thus   $h^{p,q}_i = h^{q,p}_i$, $\ell^{p,i-p}= \sum_qh^{p,q}_i$ and, by semipurity (Theorem~\ref{semipure}),  $h^{p,q}_i = 0$ if $i< n$ and $p+q > i$. 
\end{definition}
The following is \cite[Lemma 1]{SteenbrinkDB} (where we are mainly interested in the case $i=n-1$):

\begin{lemma}\label{Steenbrinkinequal}  Fix $k$ and $i$ with  $0\le k\le i \le n-1$. Then 
\begin{equation}\label{eqn2.7} \sum_{p=0}^k \ell^{i-p,p} \le 
 \sum_{p=0}^k \ell^{p,i-p}.
 \end{equation}
 Thus for example, taking $k=0$, we have $\ell^{i,0} \le \ell^{0,i}$ for all $i$, $0\le  i \le n-1$. 
  Furthermore, equality holds in {\rm(\ref{eqn2.7})}   for all $k' \le k$, i.e.\ $\displaystyle \sum_{p=0}^{k'} \ell^{i-p,p} = 
 \sum_{p=0}^{k'} \ell^{p,i-p}$  for all $k'\le k$   $\iff$ $\Gr_F^pW_jH^i(L)=0$ for all $p\le k$ and $j< i$  $\iff$ $\ell^{p,i-p}=\ell^{i-p,p}$ for all $p\le k$.
\end{lemma} 
\begin{proof} 
For a fixed $k$, we have
 $$\sum_{p=0}^k \ell^{i-p,p} =  \sum_{p=0}^k \sum_qh^{i-p, q}_i = \sum_{p=0}^k \sum_{q\leq p}h^{i-p, q}_i =  \sum_{p=0}^k \sum_{q\leq p}h^{q,i-p}_i=\sum_{\substack {i-k \le s \le i \\r+s \le i}}h^{r, s}_i,$$
by the Hodge symmetries,  whereas
 $$\sum_{p=0}^k \ell^{p,i-p} = \sum_{p=0}^k \sum_qh^{p, q}_i = \sum_{p=0}^k \sum_{p+q \le i}h^{p, q}_i=\sum_{\substack { r \le k \\r+s \le i}}h^{r, s}_i.$$
Note that, if  $i-k \le s \le i $ and $r+s \le i$, then $r\le i-s \le k$, so the second sum is greater than or equal to  the first, giving the inequality. 
Equality  holds for all $k' \le k$ $\iff$ for all $k' \le k$ , $h^{r, s}_i = 0$ for $r\leq k'$ and $s\leq i-k' -1$ (and hence automatically $r+s \le i-1 < i$)  $\iff$ $h^{r, s}_i = 0$ for $r\leq k$, $r+s \leq i-1$.  This is equivalent to: $\Gr_F^pW_jH^i(L)=0$ for $p\le k$ and $j\le i-1$.   

The final statement is clear, since $\displaystyle \sum_{p=0}^{k'} \ell^{i-p,p} = 
 \sum_{p=0}^{k'} \ell^{p,i-p}$  for all $k'\le k$   $\iff$    $\ell^{p,i-p}=\ell^{i-p,p}$ for all $p\le k$.
 \end{proof}
 
\subsection{The case of a local complete intersection} In this case, there are stronger conditions on the invariants.
 
 \begin{theorem}\label{dBzero}  If $X$ is an isolated lci, then $b^{p,q}=0$  for $q>0$,   $p+q\neq n$ or $n-1$. Moreover, the invariants  $\ell^{p,q}$ and $\dim H^q(\hX;  \Omega^p_{\hX}(\log E))$, $q> 0$, are $0$ except possibly in the following cases: $p+q= n-1$ or  $n$ or $(p,q)= (0,0)$ (for $\ell^{p,q}$) or  $(p,q)= (n,n-1)$. Finally, $\dim H^{n-1}(\hX;  \Omega^n_{\hX}(\log E))=1$. 
\end{theorem}
\begin{proof} The   statement  about $b^{p,q}$ is proved in  \cite[Theorem 5]{SteenbrinkDB}. It follows easily from a theorem of Greuel that $H^q_x(X; \Omega^p_X) = 0$ for $1\le p+q \le n-1$ \cite[Lemma 1.8]{Greuel75}, \cite[Prop. 2.3]{greuel}. Since $H^k(L) =0$ if  $k\notin\{0, n-1, n, 2n-1\}$ and $\dim H^0(L) = \dim H^{2n-1}(L) =1$, $\ell^{p,q}=0$ if $p+q\notin\{0, n-1, n, 2n-1\}$ or $p+q =2n-1$, $(p,q) \neq (n,n-1)$.  Moreover, $\ell^{0,0} = \ell^{n, n-1} =1$.  The statement  about $\dim H^q(\hX;  \Omega^p_{\hX}(\log E))$ then follows from the exact sequence of Lemma~\ref{convzero}(iii)
and the corresponding statement for $\ell^{p,q}$.
\end{proof} 

\begin{definition} For $X$  an isolated lci singularity, let $M$ be the Milnor fiber. There is a mixed Hodge structure on $H^n(M)$, which depends on the choice of a one parameter smoothing, and we set 
$$s_p = \dim\Gr_F^pH^n(M).$$
The integer $s_p$  is independent of the choice of a smoothing and is thus an invariant of the germ $(X,x)$. (More generally, for a smoothable isolated singularity $(X,x)$, one can define $s_p$ for a given smoothing component of $(X,x)$, compare \cite[proof of Theorem 2.9]{Steenbrinksemicont}.) 
\end{definition}

\begin{proposition}\label{boundlink} Let  $X$ be an isolated lci singularity of dimension $n$.
\begin{enumerate} 
\item[\rm(i)] There is an exact sequence of mixed Hodge structures connecting $M$ and $L$, given by
$$0 \to H^{n-1}(L) \to H^n(M,L) \to H^n(M) \to H^n(L) \to 0.$$
Moreover, the sequence is  self-dual. In particular $H^n(M,L)\cong (H^n(M))\spcheck(-n)$ and $H^n(L)\cong (H^{n-1}(L))\spcheck(-n)$, i.e.\   $H^n(M,L)$ and $H^n(M)$ are dual mixed Hodge structures up to a Tate twist, and similarly for  $H^n(L)$ and $H^{n-1}(L)$.
\item[\rm(ii)]  $s_{n-p}-s_p=\ell^{p,n-p-1}-\ell^{p,n-p} = \ell^{p,n-p-1}-\ell^{n-p, p-1}$. 
\item[\rm(iii)] $\ell^{p,n-p}=\ell^{n-p, p-1} \le s_p$ and $\ell^{p, n-p-1}= \ell^{n-p,p}\leq s_{n-p}$.
\end{enumerate}
\end{proposition}
\begin{proof} (i) is a standard result.  Then (ii) and  (iii) are   easy consequences  of the exactness of the sequence in (i) and the fact that morphisms of mixed Hodge structures are strict with respect to the Hodge filtration. 
\end{proof} 

We then have the following \cite[p.\ 1372]{SteenbrinkDB}:

\begin{proposition}\label{nminuspvsp}  For $X$ an isolated lci singularity of dimension $n$, in the above notation, 
$$\sum_{p=0}^ks_p \le \sum_{p=0}^k s_{n-p}.$$
If equality holds, then $\ell^{n-k-1,k} = \ell^{k+1,n-k-1} =0$. 
 \end{proposition} 
 \begin{proof}
 Since $s_{n-p}-s_p =\ell^{p,n-p-1}-\ell^{n-p,p-1}$,   
\begin{align*}
\sum_{p=0}^k(s_{n-p}-s_p) &=  \sum_{p=0}^k\ell^{p,n-p-1} - \sum_{p=0}^k\ell^{n-p,p-1} \\
&=  \sum_{r=0}^k\ell^{r,n-r-1} - \sum_{r=0}^k\ell^{n-r-1,r} + \ell^{n-k-1,k}  \ge 0, 
\end{align*}
by Lemma~\ref{Steenbrinkinequal} for the case $i=n-1$, setting $p = r+1$ and adjusting for the extra term $r=k$ in the sum. Equality holds $\iff$ $\sum_{r=0}^k\ell^{r,n-r-1} = \sum_{r=0}^k\ell^{n-r-1,r}$ and  $\ell^{n-k-1,k}=0$.
\end{proof}

\section{$k$-rational implies $k$-Du Bois}\label{section3} 

The goal of this section is to give a quick proof of Theorem~\ref{thm5} in the case of an isolated singularity, motivated by Steenbrink's proof for the case $k=0$  \cite[Proposition 3.7]{Steenbrink}. As noted in the introduction,  this result has already been proved in \cite{FL-DuBois} (by a different method, inspired by \cite{Kovacs99}), and proofs in the non-isolated lci case have been given in \cite{FL-DuBois} and \cite{MP-rat}. We include this proof as some evidence that the result might hold in general without the lci assumption. 
We begin with a result  that holds for a general isolated singularity, and which is an analogue of \cite[Lemma 2.14]{Steenbrink}:

\begin{lemma}\label{linkexseq} For all $i \ge 0$, the map
$$\mathbb{H}^i(\hX; \Omega^\bullet_{\hX}(\log E)(-E)/\sigma^{\ge k+1}) \to \mathbb{H}^i(\hX; \Omega^\bullet_{\hX}(\log E)/\sigma^{\ge k+1})$$
is injective.\end{lemma}
\begin{proof} This holds $\iff$ for all $i\ge 0$, the natural map 
$$\mathbb{H}^i(\hX; \Omega^\bullet_{\hX}(\log E)/\sigma^{\ge k+1}) \to \mathbb{H}^i(\hX; \Omega^\bullet_{\hX}(\log E)|E/\sigma^{\ge k+1})$$ is surjective. Consider the commutative diagram
$$\begin{CD}
\mathbb{H}^i(\hX; \Omega^\bullet_{\hX}(\log E )) @>{\cong}>> H^i(L) \cong  \mathbb{H}^i(\hX; \Omega^\bullet_{\hX}(\log E)|E) \\
@VVV @VVV\\
\mathbb{H}^i(\hX; \Omega^\bullet_{\hX}(\log E)/\sigma^{\ge k+1}) @>>> \mathbb{H}^i(\hX; \Omega^\bullet_{\hX}(\log E)|E/\sigma^{\ge k+1}),
\end{CD}$$
where the upper horizontal arrow is an isomorphism by Lemma~\ref{convzero}(ii). 
By Hodge theory, the right hand vertical arrow is surjective.  Hence 
$$\mathbb{H}^i(\hX; \Omega^\bullet_{\hX}(\log E)/\sigma^{\ge k+1}) \to \mathbb{H}^i(\hX; \Omega^\bullet_{\hX}(\log E)|E/\sigma^{\ge k+1})$$ is surjective as well.
\end{proof}

\begin{theorem} Let $X$ be an isolated $k$-rational singularity. Then $X$ is $k$-Du Bois.
\end{theorem}
\begin{proof} The proof is by induction on $k$. The case $k=0$ is an argument of Steenbrink (and is a special case of the arguments below). Assume inductively that the theorem has been proved for all $j< k$. In particular, we  assume  that $\Omega _X^p \to R^0\pi_*\Omega^p_{\hX}(\log E)$ is an isomorphism,   $H^q(\hX; \Omega^p_{\hX}(\log E)) = 0$ for all $p\le k$ and all $q>0$,  and that, for all $p\le k-1$,   $\Omega _X^p \to  R^0\pi_*\Omega^p_{\hX}(\log E)(-E)$ is an isomorphism and   $H^q(\hX; \Omega^p_{\hX}(\log E)(-E)) = 0$ all $q>0$. We have to show that    $\Omega _X^k \to  R^0\pi_*\Omega^k_{\hX}(\log E)(-E)$ is an isomorphism and that $H^q(\hX; \Omega^k_{\hX}(\log E)(-E)) = 0$ for all $q>0$. Note that the statement on $R^0\pi_*$ is automatic: the isomorphism $\Omega _X^k \to R^0\pi_*\Omega^k_{\hX}(\log E)$ factors as $\Omega _X^k \xrightarrow{f}  R^0\pi_*\Omega^k_{\hX}(\log E)(-E) \xrightarrow{g}  R^0\pi_*\Omega^k_{\hX}(\log E)$, where $g$ is injective. Since $g\circ f$ is an isomorphism, $f$ is an isomorphism.

 Suppose that $X$ has $k$-rational singularities. We have  the spectral sequence converging to $\mathbb{H}^i(\hX; \Omega^\bullet_{\hX}(\log E)/\sigma^{\ge k+1})$ with $E_1$ page $H^q(\hX; \Omega^p_{\hX}(\log E))$ for $p\le k$ and $0$ otherwise. By assumption, $E_1^{p,q} =0$ for $q>0$. Hence  $\mathbb{H}^{k+q}(\hX; \Omega^\bullet_{\hX}(\log E)/\sigma^{\ge k+1}) =0$ for $q > 0$. By Lemma~\ref{linkexseq},
$$\mathbb{H}^i(\hX; \Omega^\bullet_{\hX}(\log E)(-E)/\sigma^{\ge k+1}) \to \mathbb{H}^i(\hX; \Omega^\bullet_{\hX}(\log E)/\sigma^{\ge k+1})$$ is injective for $i >0$, and thus $\mathbb{H}^{k+q}(\hX; \Omega^\bullet_{\hX}(\log E)(-E)/\sigma^{\ge k+1}) =0$  if $q >0$ as well. By the inductive assumption on $k$, the only nonzero terms in the $E_1$ page of the   spectral sequence converging to  $\mathbb{H}^i(\hX; \Omega^\bullet_{\hX}(\log E)(-E)/\sigma^{\ge k+1})$ are the terms $H^0(\hX; \Omega^p_{\hX}(\log E)(-E))$ for $p\leq k$,  and $H^q(\hX; \Omega^k_{\hX}(\log E)(-E))$. By examination of the spectral sequence, it follows that, for $q>0$,
 $$H^q(\hX; \Omega^k_{\hX}(\log E)(-E)) = \mathbb{H}^{k+q}(\hX; \Omega^\bullet_{\hX}(\log E)(-E)/\sigma^{\ge k+1}) =0$$   This completes the proof.
\end{proof}

\section{The mixed Hodge structure on the Milnor fiber}\label{section4}

We begin by recalling some standard results about smoothings of isolated singularities, for which a general reference is \cite{Looijengabook}. In particular, we note the following:
\begin{enumerate}
\item If $\rho\colon (\cX, x) \to \Delta$ is a smoothing of the germ $(X,x)$, then it is possible to choose good Stein representatives for both $X$ and $\cX$.
\item If $X$ has an isolated lci singularity at $x$, then the singularity of $\cX$ is also lci \cite[Proposition 6.10]{Looijengabook}. 
\end{enumerate}
We will be interested in the case where $\rho\colon \cX \to \Delta$ is a semistable smoothing, in the terminology of \cite[Definition p.\ 520]{Steenbrink}. Thus,  $\dim \cX = n+1$ and $\cX$ has an isolated lci singularity at $x$. By assumption, there exists a log resolution $\nu\colon \hcX \to \cX$ with exceptional divisor $\cE$, such that the proper transform $\hX$ is a log resolution of $X$ with exceptional divisor $E$. Moreover, the fiber $\hcX_0$ over $0\in \Delta$ is $\hX + \cE$ (as a reduced divisor with simple normal crossings), with $\hX \cap \cE = E$.  
In our applications, we will not need to have the  precise information in (1) and (2) above.  For example, as far as (2) is concerned, for the applications in this section concerning the Milnor fiber of $X$, if $(X,x)$ is lci and locally defined inside $(\Cee^N, 0)$ by $N-n$ equations $f_1=\cdots =  f_{N-n} =0$, we can compute the Milnor fiber by taking  an explicit deformation in $(\Cee^N, 0) \times (\Delta, 0)$ defined by $f_1-t=\cdots =  f_{N-n}-t =0$, and then pass to an appropriate base change $f_1-t^A=\cdots =  f_{N-n}-t^A =0$ for some positive integer $A$. In this construction, $\cX$ clearly has an isolated lci singularity at $0$. (However, (2) is used tacitly in the statement of Theorem~\ref{semistablesmooth}.) Likewise, for (1) above, we will see that it is not necessary to make careful choices of good Stein representatives for $X$ and $\cX$.

  Let $M$ be the Milnor fiber, so that $M$ is identified as a topological space with $\cX_t$, $t\neq 0$. Moreover, $H^k(M) \neq 0$ for $k\neq 0,n$.  Let $\cL$ be the link of $\cX$, i.e.\ $\cL = \cX -\{x\} = \hcX -\cE$. Then $H^i(M)$ and $H^i(\cL)$ carry mixed Hodge structures, and we want to describe the corresponding Hodge filtrations in more detail. The following lemma  is just an exposition of \cite[(2.6)(b)]{Steenbrink}:

\begin{lemma}\label{Wang} There is an exact sequence of mixed Hodge structures:
$$\cdots  \to H^i(\cL) \to H^i(M) \xrightarrow{V} H^i_c(M)(-1) \to H^{i+1}(\cL) \to \cdots. $$
\end{lemma}
\begin{proof}  For clarity, first consider the Wang sequence for the fibration $\hcX^* \to \Delta^*$, where $\hcX^* = \hcX -\hcX_0$.   This arises by considering the sequence of relative differentials
$$0\to \rho^*\Omega^1_\Delta(\log 0)|\hcX_0 = \scrO_{\hcX_0} \to \Omega^1_{\hcX}(\log \hcX_0)|\hcX_0 \to \Omega^1_{\hcX/\Delta}(\log \hcX_0) |\hcX_0\to 0.$$
Taking exterior powers gives:
$$0\to \Omega^{\bullet -1}_{\hcX/\Delta}(\log \hcX_0)|\hcX_0 \to \Omega^\bullet_{\hcX}(\log \hcX_0)|\hcX_0 \to \Omega^\bullet_{\hcX/\Delta}(\log \hcX_0) |\hcX_0\to 0,$$
Then taking  hypercohomology essentially gives the Wang sequence for the   fiber bundle over $\Delta^*$, which is homotopy equivalent to $S^1$, with fiber $\hcX_t$: 
$$\cdots  \to H^i(\hcX^* ) \to H^i(\hcX_t) \to H^i(\hcX_t) \to H^{i+1}(\hcX^*) \to \cdots. $$ 
As it stands,   the Wang sequence does not yield an exact sequence of  mixed Hodge structures because $\hcX_0$ is not compact. Instead, define
\begin{align*}
K^\bullet &= \Omega^\bullet_{\hcX/\Delta}(\log (\hX + \cE))/ \Omega^\bullet_{\hcX/\Delta}(\log (\hX + \cE))(-\cE) =  \Omega^\bullet_{\hcX/\Delta}(\log (\hX + \cE))|\cE;\\
K^\bullet_c &= \Omega^\bullet_{\hcX/\Delta}(\log (\hX + \cE))(-\hX) /\Omega^\bullet_{\hcX/\Delta}(\log (\hX + \cE))(-\hX-\cE) = K^\bullet \otimes [\scrO_{\hcX}(-\hX)/ \scrO_{\hcX}(-\hX-\cE)].
\end{align*} 
Note that these complexes are supported on $\cE$, and there is a perfect pairing
$$K^p\otimes K_c^{n-p} \to \Omega^n_{\hcX/\Delta}(\log (\hX + \cE))(-\hX)|\cE \cong \Omega^{n+1}_{\hcX}( \hX + \cE)(-\hX)|\cE \cong\omega_{\cE}.$$
In particular, $K^\bullet$ and $K^\bullet_c$ do not depend on the choice of a good Stein representative for $\cX$ or $X$. 
It is a result of  Steenbrink \cite[(2.7)]{Steenbrink} that $H^i(M)=\mathbb{H}^i(\cE;K^\bullet)$ and $H^i_c(M)=\mathbb{H}^i(\cE;K^\bullet_c)$. Moreover, the Hodge filtrations on $H^i(M)$ and $H^i_c(M)$ are induced by the naive filtrations on $K^\bullet$ and $K^\bullet_c$ respectively. Choose coordinates so that $\cE$ is defined by $z_1\cdots z_a =0$ and $\hX$ by $z_{a+1}=0$. In degree one, $K^1$ is the quotient of $\Omega^1_{\hcX}(\log (\hX + \cE))|\cE $ by the relation $\Dis\sum_{i=1}^{a+1}\frac{dz_i}{z_i} =0$. Then the subsheaf $\Omega^1_{\hcX}(\log \cE)|\cE $ surjects onto $K^1$ and the kernel is the intersection
$\Dis\scrO_{\cE}\cdot\left[\sum_{i=1}^{a+1}\frac{dz_i}{z_i}\right] \cap \Omega^1_{\hcX}(\log \cE)|\cE$. Clearly $\Dis f\cdot \left(\sum_{i=1}^{a+1}\frac{dz_i}{z_i}\right) $ is a local section of $\Omega^1_{\hcX}(\log \cE)|\cE$ $\iff$ $\Dis f\,\frac{dz_{a+1}}{z_{a+1}}$ is holomorphic $\iff$ $f\in I_{\hX}\scrO_{\cE}$. Thus, there is an exact sequence
$$0 \to \scrO_{\hcX}(-\hX)/ \scrO_{\hcX}(-\hX-\cE) \to \Omega^1_{\hcX}(\log \cE)|\cE \to K^1 \to 0.$$
Then taking exterior powers gives 
$$0 \to K^{\bullet-1}_c \to \Omega^\bullet_{\hcX}(\log \cE)|\cE  \to K^\bullet \to 0.$$
The associated long exact sequence of hypercohomology then gives the long exact sequence of Lemma~\ref{Wang}.
\end{proof}

\begin{corollary}\label{exseqMilnorcE} Suppose that $X$ is an lci of dimension $n$. Then there is a self-dual exact sequence of mixed Hodge structures
$$0\to H^n(\cL) \to H^n(M) \xrightarrow{V} H^n_c(M)(-1) \to H^{n+1}(\cL) \to 0. $$
Moreover, if 
\begin{align*}
s_k &= \dim \Gr_F^kH^n(M) = \dim H^{n-k}(K^k);\\
 \ell^{k,n-k}(\cX) &= \dim \Gr_F^kH^n(\cL)=\dim H^{n-k}(\Omega^k_{\hcX}(\log \cE)|\cE),
 \end{align*}
  then $\ell^{k,n-k}(\cX) \le s_k$ and 
$$s_k - s_{n-(k-1)} = \ell^{k,n-k}(\cX) - \ell^{k,n-k+1}(\cX)= \dim H^{n-k}(\Omega^k_{\hcX}(\log \cE)|\cE) - \dim H^{n-k+1}(\Omega^k_{\hcX}(\log \cE)|\cE).$$
\end{corollary} 
\begin{proof}  The first statement is clear since  $H^i(M)=0$ for $i\neq 0, n$. Then, as in the proof of Proposition~\ref{boundlink}, the remaining statements  are easy consequences  of the exactness of the sequence in (i) and the fact that morphisms of mixed Hodge structures are strict with respect to the Hodge filtration, together with  the fact that  $H^n(M)$ and $H^n_c(M)(-1)$ are dual mixed Hodge structures. 
\end{proof}

\begin{corollary}\label{nminuspvsp2}  Let $X$ be an isolated lci of dimension $n$ with semistable smoothing $\cX$, and let $\ell^{p,q}(\cX)$ denote the link invariant of $\cX$. Then  $\Dis \sum_{p=0}^{k-1}s_{n-p} \le \sum_{p=0}^ks_p$. If equality holds, then $\ell^{n-k, k}(\cX) = \ell^{k+1, n-k}(\cX) = 0$. 
\end{corollary} 
\begin{proof} The proof is essentially the same as that for Proposition~\ref{nminuspvsp}. Since $\dim \cX = n+1$, 
$$s_k - s_{n-(k-1)} = \ell^{k,n-k}(\cX) - \ell^{k,n-k+1}(\cX)   = \ell^{k,n-k}(\cX) - \ell^{n-(k-1),k-1}(\cX).$$
Thus, summing over all $p\le k$,
\begin{align*}
\sum_{p\leq k}s_p - \sum_{p\leq k-1}s_{n-p} &= \sum_{p\leq k}\ell^{p,n-p}(\cX) - \sum_{p\leq k}\ell^{n-(p-1),p-1}(\cX)\\
&= \sum_{p\leq k}\ell^{k,n-k}(\cX) - \sum_{r\leq k}\ell^{n-r,r}(\cX) + \ell^{n-k, k}(\cX) \ge 0,
\end{align*}
where, as in the proof of Proposition~\ref{nminuspvsp},  we  use Lemma~\ref{Steenbrinkinequal}, but in this case  for the invariants   $\ell^{p,q}(\cX)$ and for $i = n < \dim \cX$, by setting $r = p+1$ and adjusting for the extra term $r=k$ in the sum.
\end{proof}

\begin{corollary}\label{boundk}  Let $X$ be an isolated lci of dimension $n$. Suppose that there exists a $k >   \frac12(n-1)$ such that $s_p = 0$ for all $p\le k$. Then $X$ is smooth.
\end{corollary}
\begin{proof} The given inequality on $k$ implies that $n - (k-1) = n-k + 1 \le k+1$. By Corollary~\ref{nminuspvsp2},   $s_{n-p+1} = 0$ for all $p\le k$. Thus $s_p =0$ for all $p$, $0\leq p\leq n$. Hence the Milnor number $\mu =\dim H^n(M) =0$, so $X$ is smooth.
\end{proof}

As another application of the exact sequence in Corollary~\ref{exseqMilnorcE}, we have the following formula, which is an identity involving   the Du Bois invariants, the $s_p$ and the link invariants:
 
\begin{theorem}\label{relateDBs} If $X$ is an isolated lci of dimension $n$, then for all $k\leq n-2$, 
$$b^{k, n-k} - b^{k, n-k-1} = \sum_{a=0}^{k-1}(-1)^{k-a-1}(\ell^{a, n-a} - \ell^{a, n-a -1}) +\sum_{b=0}^k(-1)^{k-b-1}(s_b - s_{n-(b-1)}).$$
\end{theorem} 

\begin{proof} First, we make a preliminary definition:

\begin{definition}\label{defn4.6}  Let $Z$ be a space, and let $\mathcal{F}$ be a sheaf of $\Cee$-vector spaces on $Z$ such that, for $i\ge a$, $H^i(Z; \mathcal{F})$ is finite dimensional, and is nonzero for only finitely many $i\ge a$. Define 
$$\chi_{\ge a}(Z; \mathcal{F}) = \chi_{\ge a}(\mathcal{F}) = \sum_{i\ge a}(-1)^i\dim H^i(Z; \mathcal{F}).$$
If $0 \to \mathcal{F}' \to \mathcal{F} \to \mathcal{F}'' \to 0$ is an exact sequence such that $H^a(Z; \mathcal{F}') \to H^a(Z; \mathcal{F})$ is injective, in particular if $H^a(Z; \mathcal{F}') =0$ or if $H^{a-1}(Z; \mathcal{F}'') =0$, then
$$\chi_{\ge a}(\mathcal{F}) = \chi_{\ge a}(\mathcal{F}') + \chi_{\ge a}(\mathcal{F}'').$$
\end{definition}

In the above situation, let $\nu \colon Z\to \overline{Z}$ be a proper morphism from a complex space $Z$ to a Stein space $\overline{Z}$ such that  $\nu$ is an isomorphism over  $\overline{Z} - \{z\}$ for some point $z\in \overline{Z}$, and let $\mathcal{F}$  be a coherent sheaf on $Z$. For $q>0$,   $H^q(Z; \mathcal{F}) = H^0(\overline{Z}; R^q\nu_*\mathcal{F})$. Thus we see that, in the definition of $\chi_{\ge a}$, as long as $a>0$, we could replace $\dim H^i(Z; \mathcal{F})$ by $\ell(R^i\nu_*\mathcal{F})$, where $\ell$ denotes the length of the skyscraper sheaf $R^i\nu_*\mathcal{F}$, and $\ell(R^i\nu_*\mathcal{F})$ only depends on the germ  $(\overline{Z},z)$. However, for the sake of clarity, we will continue to use $\dim H^i(Z; \mathcal{F})$ in the definition of $\chi_{\ge a}$. 

We want to apply  Definition~\ref{defn4.6} for the case $a = n-k-1$ to the following three exact sequences:  
 \begin{equation}\label{5.3.1}
 0 \to \Omega^k_{\hcX}(\log \cE)(-\cE) \to \Omega^k_{\hcX}(\log \cE) \to \Omega^k_{\hcX}(\log \cE)|\cE \to 0,
 \end{equation}
 the restriction sequence:
 \begin{equation}\label{5.3.2}
 0 \to \Omega^k_{\hcX}(\log(\hX + \cE))(-\hX - \cE) \to  \Omega^k_{\hcX}(\log \cE)(-\cE) \to \Omega^k_{\hX}(\log E)(-E)  \to 0,
 \end{equation}
 which is the log version of the sequence
 $$0 \to \Omega^k_{\hcX}(\log \hX  )(-\hX) \to  \Omega^k_{\hcX}  \to \Omega^k_{\hX} \to 0,$$
 twisted by $\scrO_{\hcX}(-\cE)$, and the Poincar\'e residue sequence
  \begin{equation}\label{5.3.3}
 0 \to \Omega^k_{\hcX}(\log  \cE)  \to  \Omega^k_{\hcX}(\log (\hX + \cE)) \to \Omega^{k-1}_{\hX}(\log E)  \to 0.
 \end{equation}

  In what follows, we assume that $k \le n-2$, so that $n-k-1 \ge 1$. Then $\chi_{\ge n-k -1}$ is defined for all of the sheaves in question. We claim that  $\chi_{\ge n-k -1}$ is additive for the exact sequences  ~(\ref{5.3.1})--(\ref{5.3.3}). To simplify the notation, we write $\overline{\chi}$ for $\chi_{\ge n-k -1}$. By Theorem~\ref{dBzero} applied to $X$ and to $\cX$, 
 $$H^{n-k-1}(\hcX; \Omega^k_{\hcX}(\log \cE)(-\cE)) =
 H^{n-k-1}(\hcX; \Omega^k_{\hcX}(\log \cE))  =  H^{n-k-1}(\hX;\Omega^{k-1}_{\hX}(\log E)) = 0.$$
 Thus $\overline{\chi}$ is additive for the exact sequences (\ref{5.3.1}) and (\ref{5.3.3}). 
 Moreover, by the exact sequence  ~(\ref{5.3.3}), the fact that $\scrO_{\hcX}(-\hX - \cE) \cong \scrO_{\hcX}$, and the vanishing of $H^{n-k-1}(X;\Omega^{k-1}_{\hX}(\log E))$, 
 $$H^{n-k-1}(\Omega^k_{\hcX}(\log(\hX + \cE)(-\hX - \cE)) = H^{n-k-1}(\Omega^k_{\hcX}(\log(\hX + \cE)) =0,$$ 
 so that $\overline{\chi}$ is additive for the exact sequence  (\ref{5.3.2}) as well. 
 
Thus, we see that
\begin{align*}
\overline{\chi}(\Omega^k_{\hcX}(\log \cE)) &= \overline{\chi}(\Omega^k_{\hcX}(\log \cE)(-\cE) ) + \overline{\chi}(\Omega^k_{\hcX}(\log \cE)|\cE);\\
\overline{\chi}(\Omega^k_{\hcX}(\log \cE)(-\cE)) &=  \overline{\chi}(\Omega^k_{\hcX}(\log(\hX + \cE))) + \overline{\chi}(\Omega^k_{\hX}(\log E)(-E)) ;\\
\overline{\chi}(\Omega^k_{\hcX}(\log(\hX + \cE))) &=  \overline{\chi} (\Omega^k_{\hcX}(\log  \cE)) + \overline{\chi}(\Omega^{k-1}_{\hX}(\log E)).
\end{align*}
Rewrite this as 
\begin{align*}
 \overline{\chi}(\Omega^k_{\hX}(\log E)(-E)) &= \overline{\chi}(\Omega^k_{\hcX}(\log \cE)(-\cE) - \overline{\chi}(\Omega^k_{\hcX}(\log(\hX + \cE)))\\
 &= \overline{\chi}(\Omega^k_{\hcX}(\log \cE)) - \overline{\chi}(\Omega^k_{\hcX}(\log \cE)|\cE) - \overline{\chi} (\Omega^k_{\hcX}(\log  \cE)) - \overline{\chi}(\Omega^{k-1}_{\hX}(\log E))\\
 &= - \overline{\chi}(\Omega^k_{\hcX}(\log \cE)|\cE) - \overline{\chi}(\Omega^{k-1}_{\hX}(\log E)).
 \end{align*}
 Equivalently,
 $$\overline{\chi}(\Omega^k_{\hX}(\log E)(-E)) + \overline{\chi}(\Omega^{k-1}_{\hX}(\log E)) =- \overline{\chi}(\Omega^k_{\hcX}(\log \cE)|\cE) =(-1)^{n-k-1}(s_k - s_{n-(k-1)}).$$
 We have the  exact sequence
  $$0 \to \Omega_{\hX}^{k-1}(\log E)(-E) \to \Omega_{\hX}^{k-1}(\log E) \to \Omega_{\hX}^{k-1}(\log E)|E \to 0,$$
 and $H^{n-k-1}(\hX; \Omega^{k-1}_{\hX}(\log E)(-E)) = 0$ by Theorem~\ref{dBzero}.  
 Thus
  $$\overline{\chi}(\Omega^{k-1}_{\hX}(\log E)) = \overline{\chi}(\Omega^{k-1}_{\hX}(\log E)(-E)) + \overline{\chi}(\Omega^{k-1}_{\hX}(\log E)|E).$$
 Putting this together, we get
 $$b^{k, n-k} - b^{k, n-k-1} + b^{k-1, n-k}  - b^{k-1, n-k+1}  = \ell^{k-1, n-k+1} - \ell^{k-1, n-k}- (s_k - s_{n-(k-1)}).$$
The proof of the theorem then follows by induction on $k$.
\end{proof} 

\begin{remark} (i) For $k=0$, Theorem~\ref{relateDBs} is the statement that $b^{0,n-1} = s_0$ \cite[p.\ 1372]{SteenbrinkDB}.  By \cite[Proposition 2.13]{Steenbrink} $s_n =  h^{n-1}(\scrO_{\hX})= b^{0,n-1} + \ell^{0,n-1}$. More generally, the last equation in the proof of Theorem~\ref{relateDBs} relates $s_p - s_{n-(p-1)}$ to the link invariants and Du Bois invariants of $X$. Since the difference $s_p - s_{n-p}$ is also expressed in terms of the link invariants, by Proposition~\ref{boundlink}(ii), there is a formula for $s_{p-1} - s_p$ in terms of  the link invariants and Du Bois invariants of $X$, and hence  there is such a formula for $s_p$ since $s_0 = b^{0,n-1}$. In particular, $s_p$ can be calculated solely in terms of the standard numerical invariants of $\hX$. It would be interesting to find a more direct explanation for this fact.

\smallskip
\noindent (ii) We will apply Theorem~\ref{relateDBs} in the case that $s_p=0$ for all $p\le k$. In this case, by Corollary~\ref{boundk}, we may as well assume that $k\le (n-1)/2$. Note that $(n-1)/2 \leq n-2$ $\iff$ $n\geq 3$. However,  for $n=2$, if $k\le (n-1)/2 = 1/2$, then $k = n-2 =0$ and so Theorem~\ref{relateDBs} applies in this case as well. 
\end{remark}

\section{Proofs of the main theorems}\label{section5}
  
\subsection{Characterization of $k$-Du Bois and $k$-rational singularities} 
 In this subsection, we prove Theorem~\ref{thm11} and  Theorem~\ref{Thm-improvesing}. The main point will be to relate the $k$-Du Bois property to the vanishing of certain of the $s_p$ (Theorem~\ref{thm11}(i)):
\begin{theorem}\label{mainDB0} Let $(X,x)$ be an isolated lci singularity  with Milnor fiber $M$. Then  $X$ is $k$-Du Bois $\iff$ $s_p =0$ for all $p\leq k$.  \end{theorem} 
\begin{proof} First suppose that $X$ is $k$-Du Bois. By a result of Scherk \cite[Corollary 3.11]{Scherk} for hypersurfaces and Steenbrink \cite[Theorem 1]{Steenglobal} in general, we can embed a smoothing $X_t$ of $X$ in a family of smoothings $\cY \to \Delta$ such that $Y_0$ has just one isolated singularity analytically isomorphic to $X$ and the restriction map $i^*\colon H^n(Y_t) \to H^n(M)$ is surjective. This gives an exact sequence of mixed Hodge structures
$$0 \to H^n(Y_0) \to H^n(Y_t) \to H^n(M) \to 0.$$
Since $X$, and hence $Y_0$, are $k$-Du Bois, by \cite[Corollary 1.4]{FL-DuBois} and counting dimensions, for all $p\leq k$,  the map $\Gr_F^pH^n(Y_0) \to \Gr_F^pH^n(Y_t)$ is an isomorphism and hence $\Gr_F^pH^n(M) = 0$. Thus $s_p =0$ for all $p\leq k$.

Conversely, suppose that $s_p =0$ for all $p\leq k$. By induction on $k$, we can assume that $X$ is $(k-1)$-Du Bois. In particular, by Lemma~\ref{extravan}, $b^{k,n-k} =0$.  By Corollary~\ref{nminuspvsp2}, $s_{n-p} = 0$ for all $p\le k-1$. Moreover, $\ell^{p, n-p} \leq s_p = 0$ for all $p\leq k$ and $\ell^{p, n-p-1} \leq s_{n-p} = 0$ for all $p\leq k-1$.   Then, by Theorem~\ref{relateDBs}, $b^{k,n-k-1} =0$ as well, and hence $X$ is $k$-Du Bois (see also Theorem~\ref{maindB} (iv) below).
\end{proof} 

Given Theorem~\ref{mainDB0}, it is now straightforward to give various characterizations of $k$-Du Bois and $k$-rational singularities and to prove extra vanishing statements. We begin with the $k$-Du Bois case:

\begin{theorem}\label{maindB} Let $(X,x)$ be an isolated lci of dimension $n$ and suppose $k\leq \frac12 (n-1)$. Then the following are equivalent:
\begin{enumerate}
\item[\rm(i)] $X$ is $k$-Du Bois.
\item[\rm(ii)] $\Omega_X^p\xrightarrow{\sim}\mathcal{H}^0(\uOp_X)$ is a quasi-isomorphism and $\mathcal{H}^q(\uOp_X)=0$ for $p\leq k$ and $q> 0$. 
\item[\rm(iii)] $b^{p,q}=\dim H^q(\hX;\Omega^p_{\hX}(\log E)(-E)) =0$ for $0\le p\le k$ and $q>0$.
\item[\rm(iv)] $X$ is $(k-1)$-Du Bois and $b^{k, n-k -1} =0$.
\item[\rm(v)] $b^{p,q}= \dim H^q(\hX;\Omega^p_{\hX}(\log E)(-E))=0$ in the following cases: $0\le p\le k$ and $q>0$, or $q\ge n-k-1$,  or   $p+q\neq n-1,n$.
\item[\rm(vi)] $s_p=0$ for $p\le k$.
\end{enumerate}
\end{theorem}
\begin{proof} The equivalence of (i), (ii), (iii) is Theorem~\ref{thm2} of the introduction. The equivalence (iii) $\iff$ (iv) follows immediately from Lemma~\ref{extravan}. The equivalence (i) $\iff$ (vi) is Theorem~\ref{mainDB0}. Thus, the only remaining statement to prove is the equivalence of (v) with the others. Clearly, (v) $\implies$ (iii). Conversely, assuming (iii), the condition that $b^{p,q}= 0$ for $p+q\neq n-1,n$ is automatic from Theorem~\ref{dBzero}, so we only have to show that $b^{p,q}= 0$ for $q\ge n-k-1$. Also, the only relevant case (where $p > k$) is $p=k+1, q = n-k-1$. Then $b^{k+1,n-k-1}= 0$ by Lemma~\ref{extravan}.
\end{proof}

The following is the analogue of Theorem~\ref{maindB} in the $k$-rational case, and includes Theorem~\ref{thm11}(ii).

\begin{theorem}\label{mainrat} 
Let $(X,x)$ be an isolated lci of dimension $n$, and suppose $k\leq \frac12 (n-1)$. Then the following are equivalent: 
\begin{enumerate}
\item[\rm(i)] $X$ is $k$-rational.
\item[\rm(ii)] $\Omega_X^p\xrightarrow{\sim}R^0\pi_*\Omega^p_{\hX}(\log E)$ is an isomorphism and $H^q(\hX;\pi_*\Omega^p_{\hX}(\log E))=0$ for $p\leq k$ and $q> 0$. 
\item[\rm(iii)] $H^q(\hX;\Omega^p_{\hX}(\log E)) =0$ for $0\le p\le k$ and $q>0$.
\item[\rm(iv)] $X$ is $k$-Du Bois and $\ell^{k, n-k -1} =0$.
\item[\rm(v)] $H^q(\hX;\Omega^p_{\hX}(\log E))=0$ in the following cases: $0\le p\le k$ and $q>0$, or $q\ge n-k-1$ and $(p,q) \neq (n, n-1)$,  or   $p+q\neq n-1,n$.
\item[\rm(vi)] $s_{n-p}=0$ for $p\le k$.
\end{enumerate}
\end{theorem}
\begin{proof} The equivalence of (i), (ii), and (iii) is Theorem~\ref{thm4}. 

\smallskip
\noindent (vi) $\implies$ (iv): By Proposition~\ref{nminuspvsp}, if $s_{n-p}=0$ for $p\le k$, then $s_p =0$ for all $p\le k$ and, by Proposition~\ref{boundlink}(iii),  $\ell^{k, n-k -1} \le s_{n-k}=0$. Thus, by Theorem~\ref{maindB}(vi), $X$ is $k$-Du Bois and $\ell^{k, n-k -1} =0$.

\smallskip
\noindent (iv) $\implies$ (vi): By Corollary~\ref{nminuspvsp2}, $s_{n-p} =0$ for all $p\le k-1$, so we just have to show that $s_{n-k}=0$. Again by Proposition~\ref{boundlink}(iii), $\ell^{k, n-k}\leq s_k =0$.  By Proposition~\ref{boundlink}(ii) and the hypotheses of (iv),  
$$s_{n-k} = s_{n-k} - s_k = \ell^{k, n-k -1} - \ell^{k, n-k} =0.$$

 \smallskip
\noindent  (iv) $\implies$ (iii): We  argue by induction on $k$. The case $k=0$ corresponds to $H^q(\hX; \scrO_{\hX}(-E)) =0$ for $q> 0$ and $H^{n-1}(E; \scrO_E) =0$. Then $H^{n-1}(\hX; \scrO_{\hX}) =0$ and hence $H^q(\hX; \scrO_{\hX}) =0$ for all $q> 0$. Inductively, if (iv) holds for a given $k$, then so does (vi) and hence they hold for for all $j< k$. By induction, then, if (iv) holds, then $X$ is $(k-1)$-rational. So we only need to check that  $H^q(\hX;\Omega^k_{\hX}(\log E))=0$ for all $q>0$. By Theorem~\ref{dBzero}, we can assume that $q = n-k$ or $q= n-k-1$. By the exact sequence of Lemma~\ref{convzero}(iii), $H^{n-k-1}(\hX;\Omega^k_{\hX}(\log E))\cong \Gr_F^kH^{n-1}(L)$ and $H^{n-k}(\hX;\Omega^k_{\hX}(\log E))\cong \Gr_F^kH^n(L)$. Then $\Gr_F^kH^{n-1}(L) =0$ by the hypothesis that $\ell^{k, n-k -1} =0$, and $\Gr_F^kH^n(L) =0$ because $\ell^{k, n-k} \le s_k =0$. Thus $H^q(\hX;\Omega^k_{\hX}(\log E))=0$ for all $q>0$.

\smallskip
\noindent  (iii) $\implies$ (iv): Again by induction on $k$ (starting from the case $k=-1$), we can assume that $X$ is $(k-1)$-Du Bois. Hence $b^{k,n-k}= 0$ by Lemma~\ref{extravan}. Then the exact sequence of Lemma~\ref{convzero}(iii) shows that $H^{n-k-1}(\hX;\Omega^k_{\hX}(\log E)(-E)) = \Gr^k_FH^{n-1}(L) = 0$. Thus $X$ is $k$-Du Bois and $\ell^{k, n-k -1} =0$.

\smallskip
\noindent (v) $\iff$ (iii): Clearly, (v) $\implies$ (iii). Conversely, assuming (iii), it follows from Theorem~\ref{dBzero}  that $H^q(\hX;\Omega^p_{\hX}(\log E))= 0$ for $p+q\neq n-1,n$, $(p,q) \neq (n, n-1)$ is automatic.  Thus, we need only   show that $H^q(\hX;\Omega^p_{\hX}(\log E))=0$ for $q\ge n-k-1$, $(p,q) \neq (n, n-1)$, and as in the proof of Theorem~\ref{maindB}, we can assume that $p=k+1$ and  $q= n-k-1$. Since (iii) $\iff$ (iv) has been proved, we know that $X$ is $k$-Du Bois and hence that $H^{n-k-1}(\hX;\Omega^{k+1}_{\hX}(\log E))(-E))= 0$ by Lemma~\ref{extravan}. By Lemma~\ref{convzero}(iii), it suffices to show that $\ell^{k+1, n-k-1} = 0$. This follows from (ii) in the next lemma:

\begin{lemma} Let $(X,x)$ be an isolated lci of dimension $n$.
\begin{enumerate} \item[\rm(i)] If $X$ is $k$-Du Bois and $(p,q) \neq (0,0)$ or $(n, n-1)$, then $\ell^{p,q}=0$ for $p\le k-1$, or $p=k$ and $q\neq n-k-1$.
 \item[\rm(ii)] If $X$ is $k$-rational and $(p,q) \neq (0,0)$ or $(n, n-1)$, then $\ell^{p,q}=0$ for $p\le k$, or $q\ge n-k-1$. Furthermore, the first possible non-zero link invariants satisfy:  
 $$\ell^{k+2,n-k-2}\le  \ell^{k+1,n-k-2}.$$
 \end{enumerate}
\end{lemma} 
\begin{proof} We shall just prove (ii). For $p\le k$, we have $\ell^{p, n-p} \leq s_p =0$, and $\ell^{p, n-p-1} \leq s_{n-p} =0$. As these are the only possible invariants for $\ell^{p,q}$ with $p+q\neq 0$ and $p\le k$, we see that $\ell^{p,q}=0$ for $p\le k$,  $p+q\neq 0$. If $q\ge n-k-1$  and $(p,q) \neq  (n, n-1)$, the only possible remaining nonzero invariants are $\ell^{k, n-k -1}$ and $\ell^{k+1,n-k-1}$. We have seen that $\ell^{k, n-k -1} =0$. By Lemma~\ref{Steenbrinkinequal},
$$\ell^{k+1,n-k-1} = \sum_{p\le k}\ell^{p+1,n-p-1}=\sum_{p\le k}\ell^{n-p-1, p} \le \sum_{p\le k}\ell^{p,n-p-1}=0.$$
Finally, 
$$\ell^{k+2,n-k-2} = \sum_{p\le k+1}\ell^{p+1,n-p-1}=\sum_{p\le k+1}\ell^{n-p-1, p} \le \sum_{p\le k+1}\ell^{p,n-p-1}= \ell^{k+1,n-k-2}.\qed$$
\renewcommand{\qedsymbol}{}
\end{proof}
Thus, we have established (iii) $\iff$ (v), completing the implications in Theorem~\ref{mainrat}.
\end{proof}

We summarize the connection between $k$-rational and $k$-Du Bois in the following statement (Theorem~\ref{Thm-improvesing} from the introduction):

\begin{corollary}\label{Thm-improvesing2}  Let $(X,x)$ be an isolated lci of dimension $n$.
 \begin{enumerate}
 \item[\rm(i)] If $X$ is $k$-rational, then $X$ is $k$-Du Bois. More precisely, $X$ is $k$-rational $\iff$  $X$ is $k$-Du Bois and  $\ell^{k,n-k-1}=0$.   
  \item[\rm(ii)] $X$ is $k$-Du Bois   $\iff$  $X$ is $(k-1)$-rational and $b^{k,n-k-1}=0$.  
 \end{enumerate}
 \end{corollary} 
 \begin{proof}  (i) was proved  as Theorem~\ref{mainrat}(iv). As for (ii), if $X$ is $k$-Du Bois, then   $X$ is $(k-1)$-rational by Corollary~\ref{nminuspvsp2} and Theorem~\ref{mainrat}(vi), and $b^{k,n-k-1}=0$.  Conversely, if $X$ is $(k-1)$-rational and $b^{k,n-k-1}=0$, then $X$ is $(k-1)$-Du Bois by (i), and  Theorem~\ref{maindB}(iv) then implies that $X$ is $k$-Du Bois.
 \end{proof}

\subsection{Inversion of adjunction}
Methods similar to the proof of Theorem~\ref{relateDBs} also show the following (due to Steenbrink \cite[Theorem 3.11]{Steenbrink} in case $k=0$):

\begin{theorem}\label{semistablesmooth} If $\cX \to \Delta$ is a smoothing of the isolated lci singularity $X$ which admits a semistable model, then the following are equivalent:
\begin{enumerate}
\item[\rm(i)] $X$ is $k$-Du Bois.
\item[\rm(ii)] $\cX$ is $k$-rational.
\item[\rm(iii)] $\cX$ is $k$-Du Bois and $s_k =0$.  
\end{enumerate}
\end{theorem}
\begin{proof} The proof is by a careful examination of the exact sequences (\ref{5.3.1})--(\ref{5.3.3}) in  the proof of Theorem~\ref{relateDBs}.

\smallskip
 \noindent (i) $\implies$ (iii):   Theorem~\ref{mainDB0} implies that $s_k=0$.  For the remaining statement, we argue by induction on $k$, starting with the case $k=-1$. By the $k$-Du Bois assumption on $X$,  using   (\ref{5.3.2}),
$$H^{n-k}(\hcX; \Omega^k_{\hcX}(\log(\hX + \cE))(-\hX - \cE))\xrightarrow{\cong}  H^{n-k}(\hcX;\Omega^k_{\hcX}(\log \cE)(-\cE)).$$
Since $H^{n-k-1}(\hcX;\Omega^k_{\hcX}(\log \cE)|\cE) = 0$,  the map 
$$H^{n-k}(\hcX;\Omega^k_{\hcX}(\log \cE)(-\cE)) \to H^{n-k}(\hcX;\Omega^k_{\hcX}(\log \cE))$$
is injective, using  (\ref{5.3.1}). Finally, again by the $k$-Du Bois assumption and (\ref{5.3.3}),  
$$H^{n-k}(\hcX;\Omega^k_{\hcX}(\log \cE)) \xrightarrow{\cong}  H^{n-k}(\hcX; \Omega^k_{\hcX}(\log(\hX + \cE))).$$
  Thus the natural map
$$H^{n-k}(\hcX; \Omega^k_{\hcX}(\log(\hX + \cE)(-\hX - \cE))\to H^{n-k}(\hcX; \Omega^k_{\hcX}(\log(\hX + \cE)))$$
is injective, and its image is clearly $tH^{n-k}(\hcX; \Omega^k_{\hcX}(\log(\hX + \cE)))$, where $t$ is the coordinate on $\Delta$. This implies that  $tH^{n-k}(\hcX; \Omega^k_{\hcX}(\log(\hX + \cE)))= H^{n-k}(\hcX; \Omega^k_{\hcX}(\log(\hX + \cE)))$, and thus that $H^{n-k}(\hcX; \Omega^k_{\hcX}(\log(\hX + \cE))) =0$. Hence 
$$H^{n-k}(\hcX;\Omega^k_{\hcX}(\log \cE)) = 
H^{n-k}(\hcX;\Omega^k_{\hcX}(\log \cE)(-\cE))=0.$$
 Thus $b^{k, n-k}(\cX) = 0$, and by the inductive hypotheses $\cX$ is $(k-1)$-Du Bois. Then $\cX$ is $k$-Du Bois by Theorem~\ref{maindB}(iv).    
 
 \smallskip
 \noindent (iii) $\implies$ (ii): By   Corollary~\ref{exseqMilnorcE},   $\ell^{k, n-k}(\cX) \le s_k=0$.   
 Thus $\cX$ is $k$-Du Bois of dimension $n+1$ and $\ell^{k, n-k}(\cX) =0$, hence $\cX$ is $k$-rational by Corollary~\ref{Thm-improvesing2}(i). 
 
 \smallskip
 \noindent (ii) $\implies$ (i): Again by induction, we can assume that $X$ is $(k-1)$-Du Bois, hence, via the exact sequence (\ref{5.3.3}), $H^q(\hcX; \Omega^k_{\hcX}(\log(\hX + \cE)) )=0$ for all $q>0$. Then   $H^q(\hcX; \Omega^k_{\hcX}(\log(\hX + \cE))(-\hX - \cE)) =0$ for $q>0$ as well, and $H^q(\hcX;\Omega^k_{\hcX}(\log \cE)(-\cE))=0$ since $\cX$ is $k$-Du Bois. Via the exact sequence (\ref{5.3.2}), $H^q(\hX; \Omega^k_{\hX}(\log E)(-E)) =0$ for all $q>0$, hence $X$ is $k$-Du Bois.
\end{proof}

As a corollary, we get an easy proof of the following inverse of adjunction statement (which has been proved more generally and precisely in the hypersurface case by Dirks-Musta\c{t}\u{a} \cite[Theorem 1.1]{DM}):

\begin{corollary}\label{invadj} Let $\cX$ be an isolated lci singularity such that there exists a hypersurface section $X$ of $\cX$ passing through the singular point with an isolated $k$-Du Bois singularity. Then $\cX$ is $k$-rational.
\end{corollary} 
\begin{proof} Suppose that $X = V(f)$. Then $f$ defines a $1$-parameter smoothing $\cX\to \Delta$. After a base change, we find a finite cyclic cover $\cX'$ of $\cX$ and a smoothing $\cX'\to \Delta$ of $X$ which admits a semistable model. By Theorem~\ref{semistablesmooth},  $\cX'$ is $k$-rational. So we have to prove that the same holds for $\cX$. More generally, let $G$ be a finite cyclic group and let $\ (\hcX', \cE') \to (\hcX, \cE)$ be a  $G$-equivariant log resolution. Then, as $G$ is finite cyclic, $(\Omega^p_{\hcX'}(\log \cE'))^G = \Omega^p_{\hcX}(\log \cE)$, and thus  
 $$H^q(\hcX; \Omega^p_{\hcX}(\log \cE)) =  H^q(\hcX'; (\Omega^p_{\hcX'}(\log \cE'))^G) = H^q(\hcX'; \Omega^p_{\hcX'}(\log \cE'))^G .$$
Since $\cX'$ is $k$-rational,   $H^q(\hcX; \Omega^p_{\hcX}(\log \cE)) =0$   for $p\le k$ and $q>0$. Hence   $\cX$ is $k$-rational.
 \end{proof}

\section{The hypersurface case}\label{section6}
For the remainder of this paper, we assume unless otherwise stated that $X$ is an isolated hypersurface singularity. In this case, we can relate the previous invariants $s_p$ to the so-called spectrum of $X$. We briefly review this connection and then analyze the case where $X$ is $k$-Du Bois but not $k$-rational.  
 \subsection{Some remarks on hypersurface singularities}  In the hypersurface case, $H^n(M)$ is a direct sum of eigenspaces $H^n(M)(a)$, $a\in \Q$ and $-1 < a \leq 0$, where the semisimple part of the monodromy acts on $H^n(M)(a)$ with eigenvalue $e^{-2a\pi \sqrt{-1}}$. 
There is a related invariant as follows: suppose that $(X,0)\subseteq (\Cee^{n+1},0)$ is defined by the function $f$, and define $Q_f =\Omega^{n+1}_{\Cee^{n+1}, 0}/df\wedge \Omega^n_{\Cee^{n+1}, 0}$. Then $\dim Q_f =\mu$. The choice of a nonvanishing holomorphic $(n+1)$-form identifies $Q_f$ with the Milnor algebra $\scrO_{\Cee^{n+1}, 0}/J(f) $, where $J(f) =\Dis \left(\frac{\partial f}{\partial z_1}, \dots, \frac{\partial f}{\partial z_{n+1}} \right)$ is the Jacobian ideal of $f$. There is the (decreasing) Kashiwara-Malgrange \textsl{$V$-filtration} on $Q_f$, indexed by rational numbers $b\in \Q$. The invariants $Q_f$ and $H^n(M)$ are connected as  follows \cite[proof of Theorem 7.1]{ScherkSteenbrink}:

\begin{theorem} There is an isomorphism
$$\Gr_V^bQ_f = V_bQ_f/V_{>b}Q_f \cong  \Gr_F^pH^n(M)(a),$$
where $b= n-p +a$, and hence $p$ is the unique integer for which  $-1 < a \leq 0$. \qed
\end{theorem}

 Define $d(b) = \dim \Gr_F^pH^n(M)(a) = \dim \Gr_V^bQ_f$. Thus $\sum_{b\in \Q}d(b) = \mu$. It will be convenient to shift by $1$: define the spectrum 
$$\Sp(X,x)= \{\alpha \in \Q: d(\alpha - 1) \neq 0\}$$
and let $m_\alpha = d(\alpha -1)$. Hence $\Gr_V^bQ_f \neq 0$ $\iff$ $b+1\in \Sp(X,x)$.  
With this normalization, 
$\Sp(X,x)\subseteq (0,n+1)\cap\Q$, and  $\Sp(X,x)$ is symmetric about $\Dis\frac{n+1}{2}$  in the sense that $\alpha \in  \Sp(X,x)$ $\iff$ $n+1 -\alpha \in \Sp(X,x)$, and in fact $m_{n+1-\alpha} = m_\alpha$. Then   $\sum_\alpha m_\alpha=\mu$, and more precisely:

\begin{lemma}\label{somespec} In the above notation, for all $d\geq 0$,
\begin{align*} 
s_{n-p}&=\sum_{\alpha\in (p,p+1]\cap \Sp(X,x)} m_\alpha; & \sum_{p=0}^d s_{n-p}& =\sum_{\alpha\in (0,d+1]\cap \Sp(X,x)} m_\alpha;\\
  s_p& =\sum_{\alpha\in [p,p+1)\cap \Sp(X,x)} m_\alpha; 
 &\sum_{p=0}^d s_p&=\sum_{\alpha\in [0,d+1)\cap \Sp(X,x)} m_\alpha.
 \end{align*}
 Finally, $s_{n-p}-s_p=m_{p+1}-m_{p}$. 
 \end{lemma}
 \begin{proof} The first two equalities  are immediate from the definitions and the isomorphism $\Gr_V^bQ_f \cong  \Gr_F^pH^n(M)(a)$. The second two equalities follow from the fact that  $\alpha \in (n-p, n-p+1]$ $\iff$ $n+1-\alpha \in [p, p+1)$ and the symmetry $m_{n+1-\alpha} = m_\alpha$. The last statement follows from the formulas for $s_p$ and $s_{n-p}$. 
 \end{proof}
 
 \begin{definition} Following Saito \cite{Saito-b}, we define the  \textsl{minimal exponent} $\widetilde{\alpha}_{X,x}= \widetilde{\alpha}_X$ to be the minimal spectral number $\alpha_{\min}=\min\{\alpha: \alpha\in\Sp(X,x)\}$. Equivalently, if $b < \widetilde{\alpha}_X-1$, then $\Gr_V^bQ_f = 0$, and $\Gr_V^bQ_f \neq 0$ for $b= \widetilde{\alpha}_X-1$.
 \end{definition}
 
 \begin{remark} In fact, this definition is equivalent to Definition~\ref{def6} in the isolated case, cf.\ \cite[Remark 1.4 (iv), (v)]{KLS}. While $\tilde\alpha_X$ is defined more generally for hypersurface singularities, in the non-isolated case Saito has remarked that it is not in general determined by the spectrum.
 \end{remark} 
 
   By \cite{Saito-b},  $\widetilde{\alpha}_X >1$ $\iff$ $X$ is a rational singularity, and $\widetilde{\alpha}_X \geq 1$ $\iff$ $X$ is a Du Bois singularity. We can generalize this as follows:

\begin{lemma}\label{spectralchar}
Let $(X,x)$ be an isolated hypersurface singularity. Then 
\begin{itemize}
\item[\rm(i)] $\widetilde\alpha_X \geq k+1$  $\iff$ $s_p=0$ for $0\le p\le k$.
\item[\rm(ii)] $\widetilde\alpha_X > k+1$  $\iff$ $s_{n-p}=0$ for $0\le p\le k$. 
\end{itemize} 
\end{lemma}
\begin{proof} This is immediate from Lemma~\ref{somespec}.
\end{proof}

\begin{corollary} Let $(X,x)$ be an isolated hypersurface singularity. Then 
\begin{itemize}
\item[\rm(i)] $X$ is $k$-Du Bois   $\iff$ $\widetilde\alpha_X \geq k+1$.
\item[\rm(ii)] $X$ is $k$-rational  $\iff$ $\widetilde\alpha_X > k+1$. \qed
\end{itemize} 
\end{corollary}

\begin{remark} As noted in the introduction (i) holds for a general, not necessarily isolated, hypersurface singularity by \cite[Thm. 1]{JKSY-duBois}  and \cite[Thm. 1.1.]{MOPW}, and (ii) holds under the same assumption   by \cite[Appendix]{FL-DuBois}  and  \cite{MP-rat}.
\end{remark}

 \begin{corollary}\label{wtdhom}  Let $a_1, \dots , a_{n+1}$ be positive integers, and let $\Cee^*$ act on $\Cee^{n+1}$ with the weights $a_1, \dots , a_{n+1}$, i.e.\ by the action $\lambda \cdot (z_1, \dots, z_{n+1}) = (\lambda^{a_1}z_1, \dots, \lambda^{a_{n+1}}z_{n+1})$. Suppose that $f(z)\in \Cee[z_1, \dots, z_{n+1}]$ satisfies $f(\lambda \cdot z) = \lambda^df(z)$ and $\{f=0\}$ defines an isolated singularity of dimension $n\geq 2$. Finally, set $w_i = a_i/d$. Then:
\begin{enumerate}
 \item[\rm(i)] {\rm(Saito)} $X$ is $k$-Du Bois $\iff$ $\sum_{i=1}^{n+1}w_i \ge k+1$. 
 \item[\rm(ii)] $X$ is $k$-rational $\iff$ $\sum_{i=1}^{n+1}w_i > k+1$.
 \end{enumerate}
\end{corollary}
\begin{proof} Recall that $Q_f \cong \Cee[z_1, \dots, z_{n+1}]/J(f)$, where $J(f)$ is the Jacobian ideal of $f$. Moreover, for a multi-index $\alpha = (\alpha_1, \dots, \alpha_{n+1})$ of nonnegative integers and monomial $z^\alpha = z_1^{\alpha_1}\cdots z_{n+1}^{\alpha_{n+1}}$, set $\ell(\alpha) = \sum_{i=1}^{n+1}(\alpha_i +1)w_i$. Then, by e.g.\ \cite[Example p.\ 651]{ScherkSteenbrink}, the image of  $z^\alpha$ in $Q_f$ lies in  $V^bQ_f$, with $b= \ell(\alpha) - 1$. In particular, the minimal possible $b$ is $\ell(0) - 1 = \sum_{i=1}^{n+1}w_i -1$, and hence we obtain the following (due to Saito  \cite[(2.5.1)]{SaitoV}): $\widetilde{\alpha}_X =b+1 = \sum_{i=1}^{n+1}w_i$. Thus $X$ is $k$-Du Bois $\iff$ $\widetilde{\alpha}_X  =\sum_{i=1}^{n+1}w_i \ge k+1$, and $X$ is $k$-rational $\iff$ $\widetilde{\alpha}_X = \sum_{i=1}^{n+1}w_i > k+1$. 
 \end{proof}
 
 \begin{remark} It is straightforward to extend the results of Corollary~\ref{wtdhom} to the case of a positive weight deformation of an isolated weighted homogeneous hypersurface singularity, using the result of Varchenko \cite{Varchenko82} that such a deformation is a $\mu = $ constant deformation and the theorem of Steenbrink \cite{Steenbrinksemicont} that the spectrum is semicontinuous in an appropriate sense.
 \end{remark}

 \subsection{$k$-liminal singularities}  
 
 \begin{definition} Let $X$ be  the germ of an isolated singularity. The space $X$ is \textsl{$k$-liminal} if $X$ is $k$-Du Bois but not $k$-rational. Thus, for a nonnegative integer $k$, if $X$ is an isolated hypersurface singularity, it follows from Theorem~\ref{maindB} and Theorem~\ref{mainrat} that $X$ is $k$-liminal $\iff$ $\widetilde{\alpha}_X =k+1$. 
 \end{definition}
 
We analyze the property that $X$ is $k$-liminal in more detail in the hypersurface case. First, we recall some general results: We note the following theorem of Dimca-Saito \cite[\S4.11]{DS} (cf.\ also \cite[Remark 3.7]{Saito-Brieskorn2}):

\begin{theorem}\label{multone} For   an isolated hypersurface singularity $X$, if $b= \widetilde{\alpha}_X-1$, then $\dim \Gr_V^bQ_f= \dim \Gr_V^{\widetilde{\alpha}_X-1}Q_f =1$. Equivalently, $m_{\widetilde{\alpha}_X} = 1$. Moreover, $V_{>b}Q_f =\mathfrak{m}_xQ_f$. \qed
\end{theorem}

\begin{corollary}\label{boundalpha} If $X$ is an isolated hypersurface singularity of dimension $n$, then $\widetilde{\alpha}_X \leq \Dis \frac{n+1}{2}$, and $\widetilde{\alpha}_X = \Dis \frac{n+1}{2}$ $\iff$ $X$ is an ordinary double point. In particular, if $n=2k+1$ is odd, then there is no isolated hypersurface singularity $X$ of dimension $n$ which is $k$-rational, and $X$ is $k$-Du Bois $\iff$ $X$ is $k$-liminal $\iff$ $X$ is an ordinary double point.
\end{corollary}
\begin{proof} The first statement follows from the symmetry of the spectrum about $(n+1)/2$. (It also follows from Corollary~\ref{boundk} and Lemma~\ref{spectralchar}.)  For the second, again by symmetry, if $\widetilde{\alpha}_X = \Dis \frac{n+1}{2}$, then $Q_f = \Gr_V^{n-1/2}Q_f$ has dimension $1$. Hence $J(f) =\mathfrak{m}_x$, i.e.\ the partial derivatives of $f$ generate the maximal ideal in $\Cee\{z_1,\dots, z_{n+1}\}$. Thus $X$ is an ordinary double point.
\end{proof}

\begin{remark} (i) This statement has been generalized to the case where $X$ is not assumed \emph{a priori} to have isolated singularities by Dirks-Musta\c{t}\u{a} \cite[Corollary 6.3]{DM}. There is also a sharper bound due to Musta\c{t}\u{a}-Popa \cite[Theorem E(3)]{MP-V}. 

\noindent (ii) The case $\dim X =3$ and  $k=1$ in the final statement of Corollary~\ref{boundalpha} is due to Namikawa-Steenbrink \cite[Theorem 2.2]{NS}.
\end{remark}

The above results then imply  the following corollary:

\begin{corollary}\label{linkone} If $X$ is an isolated hypersurface $k$-liminal singularity of dimension $n$, then $s_{n-k} = \ell^{n-k,k} =\dim \Gr_F^{n-k}H^n(L) = 1$ and the map $\Gr_F^{n-k}H^n(M) \to \Gr_F^{n-k}H^n(L)$ is an isomorphism. Finally, $\ell^{k,n-k-1} = \dim \Gr_F^kH^{n-1}(L) =1$, the natural map
$H^{n-k-1} (\hX;\Omega_{\hX}^k(\log E)) \to \Gr_F^kH^{n-1}(L)$ is an isomorphism, and hence $\dim H^{n-k-1} (\hX;\Omega_{\hX}^k(\log E)) =1$.
\end{corollary}
\begin{proof} By Theorem~\ref{mainDB0},  $s_k =0$, hence $\ell^{k, n-k}= 0$. Theorem~\ref{multone} implies that $m_{\widetilde{\alpha}_{X,x}} =m_{k+1} =1$. By    Lemma~\ref{somespec},  $s_{n-k}= s_{n-k} - s_k =  m_{k+1} - m_k  =m_{k+1} =1$.   Moreover, by Proposition~\ref{boundlink}(ii),
$$1 = s_{n-k} - s_k = \ell^{k,n-k-1}-\ell^{k,n-k} = \ell^{k,n-k-1}.$$
Thus $\ell^{n-k,k} =\ell^{k,n-k-1} =  1$. Since   the map $\Gr_F^{n-k}H^n(M) \to \Gr_F^{n-k}H^n(L)$ is surjective, and both sides have dimension one, it is an isomorphism. The final statement  follows from Lemma~\ref{convzero}(iii) and the fact that, since $X$ is $k$-Du Bois, $H^{n-k-1} (\hX;\Omega_{\hX}^k(\log E)(-E))= H^{n-k} (\hX;\Omega_{\hX}^k(\log E)(-E)) =0$.
\end{proof}

\bibliography{duBois}
\end{document}